\documentclass[hidelinks, 10pt]{amsart}

\usepackage{amssymb}
\usepackage{relsize}
\usepackage[bbgreekl]{mathbbol}
\usepackage{amsfonts}

\usepackage{cite}
\usepackage{hyperref}
\usepackage[all]{xy}
\usepackage{enumerate}
\usepackage{bbm}
\usepackage{tikz-cd}
\usepackage{extpfeil}
\usepackage{comment}
\usepackage{diagbox}

\newcommand{\bA}{{\mathbb A}}

\newcommand{\bC}{{\mathbb C}}

\newcommand{\bG}{{\mathbb G}}

\newcommand{\bL}{{\mathbb L}}
\newcommand{\bM}{{\mathbb M}}
\newcommand{\bN}{{\mathbb N}}

\newcommand{\bP}{{\mathbb P}}
\newcommand{\bQ}{{\mathbb Q}}

\newcommand{\bZ}{{\mathbb Z}}

\newcommand{\cA}{{\mathcal A}}

\newcommand{\cC}{{\mathcal C}}

\newcommand{\cE}{{\mathcal E}}
\newcommand{\cF}{{\mathcal F}}

\newcommand{\cO}{{\mathcal O}}

\newcommand{\fD}{{\mathfrak D}}

\newcommand{\fX}{{\mathfrak X}}

\newcommand{\ofX}{\overline{\mathfrak{X}}}

\newcommand{\fm}{{\mathfrak m}}

\newcommand{\oK}{\overline{K}}
\newcommand{\ok}{\overline{k}}
\newcommand{\ox}{\overline{x}}
\newcommand{\oy}{\overline{y}}

\newcommand{\obQ}{\overline{\mathbb{Q}}}

\newcommand{\tf}{\tilde{f}}

\newcommand{\tx}{\widetilde{x}}

\newcommand{\et}{\mathrm{\acute{e}t}}

\DeclareMathOperator{\rk}{{rk}}

\DeclareMathOperator{\End}{{End}} 

\DeclareMathOperator{\Hom}{{Hom}} 
\DeclareMathOperator{\Ind}{{Ind}}

\DeclareMathOperator{\id}{{id}}

\DeclareMathOperator{\op}{{op}}

\DeclareMathOperator{\Spec}{{Spec}}

\DeclareMathOperator{\ev}{{ev}}

\DeclareMathOperator{\Func}{Func}
\DeclareMathOperator{\cont}{cont}
\DeclareMathOperator{\Proj}{Proj}

\DeclareMathOperator{\Gal}{{Gal}}

\newcommand{\alg}{\mathrm{alg}}

\newcommand{\Fr}{\operatorname{Fr}}

\newcommand{\fin}{\mathrm{fin}}

\newcommand{\red}{\mathrm{red}}
\newcommand{\sm}{\mathrm{sm}}

\newcommand{\oX}{\overline{X}}

\newcommand{\oC}{\overline{C}}

\newcommand{\oF}{\overline{F}}

\newcommand{\pro}{\mathrm{pro}}
\newcommand{\FEt}{\mathrm{F\acute{E}t}}

\newcommand{\PF}{\bP^1_F\setminus\{0,1,\infty\}}
\newcommand{\PoF}{\bP^1_{\oF}\setminus\{0,1,\infty\}}

\newcommand{\Apio}{\bQ_p[\pi_1^{\pro-\alg}(\PoF,0_v)]}

\newcommand{\algfun}[1]{\bQ_p[\pi_1^{\pro-\alg}(#1)]^{G_F-\fin}}

\DeclareSymbolFontAlphabet{\mathbb}{AMSb}
\DeclareSymbolFontAlphabet{\mathbbl}{bbold}

\makeatletter

\newcommand{\Rmnum}[1]{\expandafter\@slowromancap\romannumeral #1@}
\makeatother

\newtheorem{pr}{Proposition}[section]

\newtheorem{thm}[pr]{Theorem}
\newtheorem{conj}[pr]{Conjecture}

\newtheorem{lm}[pr]{Lemma}

\newtheorem{cor}[pr]{Corollary}

\theoremstyle{definition}
\newtheorem{rem}[pr]{Remark}

\theoremstyle{definition}
\newtheorem{example}[pr]{Example}
\newtheorem{question}[pr]{Question}

\numberwithin{equation}{section}

\begin{document}

\title[Universality of the Galois action on $\pi_1$ of $\bP^1\setminus\{0,1,\infty\}$]
{Universality of the Galois action on the fundamental group of $\bP^1\setminus\{0,1,\infty\}$}

\author[]{Alexander Petrov}
\address{Institute for Advanced Study, USA}
\email{alexander.petrov.57@gmail.com}

\begin{abstract}
We prove that any semi-simple representation of the Galois group of a number field coming from geometry appears as a subquotient of the ring of regular functions on the pro-algebraic completion of the fundamental group of the projective line with $3$ punctures.
\end{abstract}

\maketitle

\section{Introduction}

A surprising result of Belyi's \cite{belyi} says that every non-unit element of the absolute Galois group $G_{\bQ}$ acts non-trivially on the \'etale fundamental group $\pi_1^{\et}(\bP^1_{\obQ}\setminus\{0,1,\infty\})$ of the projective line with $3$ punctures. It can be deduced from this that every finite image representation of the Galois group can be found in the space of locally constant functions on that fundamental group:

\begin{pr}[Proposition \ref{artin motives}]\label{intro belyi}
For a number field $F$, any continuous finite image representation $\rho:G_F\to GL_d(\bQ)$ can be embedded into the space of locally constant functions $\Func^{\mathrm{loc.}\mathrm{const.}}(\pi_1^{\et}(\PoF,0_v),\bQ)$. Here $0_v$ is a tangential base point supported at $0$.
\end{pr}

In this paper we generalize this result by proving that {\it every} semi-simple representation coming from geometry appears as a subquotient of the space of functions on the pro-algebraic completion of $\pi_1^{\et}(\PoF,0_v)$. Fix a prime $p$. Explicitly, the space of regular functions $\bQ_p[\pi_1^{\pro-\alg}(\PoF,0_v)]$ is the space of continuous functions $\pi_1^{\et}(\PoF,0_v)\to \bQ_p$ that can be factored as $\pi_1^{\et}(\PoF,0_v)\xrightarrow{\rho} GL_n(\bQ_p)\xrightarrow{f}\bQ_p$ where $\rho$ is a continuous representation and $f\in\bQ_p[GL_{n,\bQ_p}]$ is a regular function. Denote by $\bQ_p[\pi_1^{\pro-\alg}(\PoF,0_v)]^{G_F-\fin}\subset \bQ_p[\pi_1^{\pro-\alg}(\PoF,0_v)]$ the subspace of functions whose $G_F$-orbit spans a finite-dimensional space. This is our main result:

\begin{thm}\label{main theorem intro}
For any separated scheme $X$ of finite type over a number field $F$ and any $i\in \bN$, the semi-simplification of the $G_F$-representation $H^i(X_{\oF},\bQ_p)$ appears as a subquotient of the space $\bQ_p[\pi_1^{\pro-\alg}(\PoF,0_v)]^{G_F-\fin}$.
\end{thm}

Conversely, it was shown in \cite[Corollary 8.6]{petrov} that for any smooth variety $Y$ over $F$ with an $F$-rational base point $y$, any finite-dimensional subrepresentation $V$ of $\bQ_p[\pi_1^{\pro-\alg}(Y_{\oF},\oy)]$ is de Rham at places above $p$ and is almost everywhere unramified. Therefore, the Fontaine-Mazur conjecture \cite{fontaine-mazur} is equivalent to the conjunction of the following two conjectures, see Lemma \ref{fm restatement}:

\begin{conj}\label{conj geom in pi1}
Every irreducible finite-dimensional representation of $G_F$ that appears as a subquotient of $\obQ_p[\pi_1^{\pro-\alg}(\PoF,0_v)]^{G_F-\fin}$ is a subquotient of $H^i_{\et}(X_{\oF},\obQ_p(j))$ for some smooth projective variety $X$ and $i\geq 0,j\in\bZ$.
\end{conj}

\begin{conj}\label{conj de rham in pi1}
Any irreducible $\obQ_p$-representation of $G_F$ that is almost everywhere unramified and is de Rham at places above $p$ can be established as a subquotient of $\obQ_p[\pi_1^{\pro-\alg}(\PoF,0_v)]^{G_F-\fin}$ for every tangential base point $0_v$ supported at $0$.
\end{conj}
We will observe in Corollary \ref{weil numbers over number field}, extending a result of Pridham \cite{pridhamweight}, that for every Galois representation appearing in $\obQ_p[\pi_1^{\pro-\alg}(\PoF,0_v)]^{G_F-\fin}$ the Frobenius eigenvalues at almost all places are Weil numbers, a condition notably absent from the statement of the Fontaine-Mazur conjecture.

Before sketching the proof of the theorem, let us get a feel for working with the Galois action on the pro-algebraic completion of the \'etale fundamental group by looking at two mechanisms for producing Galois representations inside $\bQ_p[\pi_1^{\pro-\alg}(Y_{\oF},y)]$, for a variety $Y$ over $F$. In Example \ref{intro: irreducible monodromy example} geometrically irreducible local systems yield Galois representations inside functions on the fundamental group, and Example \ref{intro: belyi example} demonstrates how Belyi's theorem implies Theorem \ref{main theorem intro} when $X$ is a curve.

\begin{example}\label{intro: irreducible monodromy example} If $\bL$ is a $\bQ_p$-local system then the corresponding representation of the geometric fundamental group defines a morphism $\rho^{\mathrm{geom}}_{\bL}:\pi_1^{\pro-\alg}(Y_{\oF},y)\to GL_{\bL_{y}}$ to the algebraic group of invertible matrices on the space $\bL_y$. Regular functions on $GL_{\bL_y}$ then give rise to elements of $\bQ_p[\pi_1^{\pro-\alg}(Y_{\oF},y)]$. In particular, there is a $G_F$-equivariant map $\End(\bL_y)\to \bQ_p[\pi_1^{\pro-\alg}(Y_{\oF},y)]$ whose image consists of {\it matrix coefficients} of the representation $\rho_{\bL}^{\mathrm{geom}}$; it is the space dual to the $\bQ_p$-span of the image of the map $\pi_1^{\et}(Y_{\oF},y)\to \End(\bL_y)$. For example, if $\bL|_{Y_{\oF}}$ is absolutely irreducible, the image of $\rho_{\bL}^{\mathrm{geom}}$ spans all of $\End(\bL_y)$ by Burnside's theorem and thus the map $\End(\bL_y)\to \bQ_p[\pi_1^{\pro-\alg}(Y_{\oF},y)]$ is an inclusion. Thus, if a Galois representation $V$ can be established as the fiber over $y$ of a geometrically absolutely irreducible local system on $Y$, then the adjoint representation $V\otimes V^{\vee}$ of the Galois group $G_F$ is a subspace of the ring of regular function on the pro-algebraic completion of $\pi_1^{\et}(Y_{\oF},y)$.
\end{example}
\begin{example}\label{intro: belyi example} Suppose that $C$ is a smooth projective curve equipped with a finite morphism $f:C\to \bP^1_F$ that is \'etale over $\PF$. Belyi's theorem \cite[Theorem 4]{belyi} says that for any curve over $F$ one can choose such a morphism. Assume further that $C$ contains a rational point $x\in C(F)$ with $f(x)=0$. The fundamental group of the open subscheme $U:=f^{-1}(\PF)\subset C$ is then a finite index subgroup $f_*(\pi_1^{\et}(U_{\oF},x_{v'}))\subset \pi_1^{\et}(\PoF,0_v)$ where $x_{v'}$ and $0_v$ are appropriate tangential base points. By Lemma \ref{pro-alg: finite index surjection} the restriction to this finite index subgroup induces a $G_F$-equivariant surjection $\bQ_p[\pi_1^{\pro-\alg}(\PoF,0_v)]\to \bQ_p[\pi_1^{\pro-\alg}(U_{\oF},x_{v'})]$. On the other hand, the map $\pi_1^{\et}(U_{\oF},x_{v'})\to H^1_{\et}(U_{\oF},\bQ_p)^{\vee}$ yields a surjective map $\pi_1^{\pro-\alg}(U_{\oF},x_{v'})\to \underline{H^1_{\et}(U_{\oF},\bQ_p)^{\vee}}$ onto the corresponding vector group. The linear functions on that vector group then give a subspace $H^1_{\et}(U_{\oF},\bQ_p)\subset \bQ_p[\pi_1^{\pro-\alg}(U_{\oF},x_{v'})]$. In particular, this establishes $H^1_{\et}(C_{\oF},\bQ_p)$ as a subrepresentation of $\bQ_p[\pi_1^{\pro-\alg}(\PoF,0_v)]$.
\end{example}

The only arithmetic input needed for our proof is Belyi's theorem and the rest is a purely algebro-geometric argument that we will now describe. A related result has been recently independently obtained by Joseph Ayoub: it follows from \cite[Corollary 4.47]{ayoub} that the action of the motivic Galois group of $\bQ$ on the motivic fundamental group of $\bP^1_{\obQ}\setminus\{0,1,\infty\}$ is faithful. Our Proposition \ref{intro tensor product} can be used to deduce Theorem \ref{main theorem intro} from this faithfulness result, though, to the best of my understanding, this would give a proof different from ours; in particular, our argument is constructive in that it gives an explicit way of finding a given Galois representation $H^i_{\et}(X_{\oF},\bQ_p)$ inside $\bQ_p[\pi_1^{\pro-\alg}(\PoF,0_v)]^{G_F-\fin}$. 

Denote by $\cC_F$ the set of finite-dimensional representations of $G_F$ that can be realized as subquotients of $\bQ_p[\pi_1^{\pro-\alg}(\PoF,0_v)]^{G_F-\fin}$, for every choice of the tangential base point $0_v$ supported at $0$. To prove the theorem, we will show first that $\cC_F$ is closed under direct sums and tensor products (in particular, every representation from $\cC_F$ appears in $\bQ_p[\pi_1^{\pro-\alg}(\PoF,0_v)]$ with an arbitrarily large multiplicity).

\begin{pr}[Proposition \ref{tensor: tensor product}]\label{intro tensor product}
For any two Galois representations $V_1,V_2\in\cC_F$ the representations $V_1\oplus V_2$ and $V_1\otimes V_2$ also belong to $\cC_F$. 
\end{pr}

This is a special feature of the variety $\bP^1_{\oF}\setminus\{0,1,\infty\}$ which comes down to the fact that the Cartesian square of its fundamental group can be established as a subquotient of the fundamental group itself, compatibly with the Galois actions.

The proof of Theorem \ref{main theorem intro} now proceeds by induction on the dimension of $X$. The base case $\dim X=0$ is given by Proposition \ref{intro belyi}. Assuming that the theorem has been proven for all schemes of dimension $<\dim X$, using resolution of singularities and the Gysin sequence, we may freely replace $X$ by a birational variety. We can therefore assume that $X$ admits a smooth proper morphism to a (possibly open) curve. Applying Belyi's theorem to this curve we may moreover assume that $X$ admits a smooth proper morphism $f:X\to \PF$ to the projective line with three punctures.

Leray spectral sequence together with Artin vanishing now tell us that in order to show that the semi-simplification of $H^n(X_{\oF},\bQ_p)$ lies in $\cC_F$ it is enough to do so for the Galois representations $H^0(\PoF,R^{n}\pi_*\bQ_p)$ and $H^1(\PoF,R^{n-1}\pi_*\bQ_p)$. The statement about $0$th cohomology is immediate from the induction assumption, because $H^0(\PoF,R^{n}\pi_*\bQ_p)$ embeds into a stalk $(R^n\pi_*\bQ_p)_y=H^n(f^{-1}(y)_{\oF},\bQ_p)$. The assertion about $1$st cohomology is proven using the following purely algebraic observation

\begin{pr}[Proposition \ref{H1 main}]\label{intro H1}
For a $\bQ_p$-local system $\bL$ on any geometrically connected finite type scheme  $Y$ over $F$ equipped with a base point $y$, the Galois representation $H^1_{\et}(Y_{\oF},\bL)$ is a subquotient of the tensor product $\bQ_p[\pi_1^{\pro-\alg}(Y_{\oF},y)]^{G_F-\fin}\otimes \bL_y$.
\end{pr}

This proves that $H^1(\PoF,R^{n-1}\pi_*\bQ_p)$ is in $\cC_F$ because $\cC_F$ is stable under tensor products, and this finishes the proof of the induction step.

Proposition \ref{intro H1} crucially uses matrix coefficients of non-semi-simple representations of $\pi_1^{\et}(\PoF,0_v)$ and the analogous statement is false for the pro-reductive completion of $\pi_1^{\et}(Y_{\oF},y)$. This begs the question:

\begin{question}\label{reductive question intro}
Which representations of $G_F$ appear as subquotients of the space $\bQ_p[\pi_1^{\pro-\mathrm{red}}(\PoF,0_v)]^{G_F-\fin}$ of regular functions on the pro-reductive completion of $\pi_1^{\et}(\PoF,0_v)$?
\end{question}

More explicitly, this question can be reformulated as asking to classify, for all finite extensions $F'\supset F$, representations of $G_{F'}$ having the form $V\otimes V^{\vee}$ where $V$ is the stalk at $0_v$ of a geometrically irreducible $\obQ_p$-local system on $\bP^1_{F'}\setminus\{0,1,\infty\}$, cf. Lemma \ref{reductive explicit}.
%As we already mentioned, all subquotients of $\bQ_p[\pi_1^{\pro-\alg}(\PoF,0_v)]^{G_F-\fin}$ (this includes all subquotients of $\bQ_p[\pi_1^{\pro-\red}(\PoF,0_v)]^{G_F-\fin}$ as well) are almost everywhere unramified and are de Rham at places above $p$. Moreover, as was observed by Pridham \cite{pridhamweight}, the results of L. Lafforgue imply that for every finite-dimensional $G_F$-representation inside $\bQ_p[\pi_1^{\pro-\red}(\PoF,0_v)]$ or $\bQ_p[\pi_1^{\pro-\alg}(\PoF,0_v)]$ the eigenvalues of almost all Frobenius elements are Weil numbers. 

Lastly, let us remark that the usage of tangential base points is important for our proof, but Theorem 1.2 might well be true for classical base points as well. We comment on this in Subsection \ref{base points subsection}, see also Corollary \ref{weil numbers over number field} for an instance of a substantial difference between the Galois action on the fundamental group with respect to a tangential base point and a classical base point. 

{\bf Notation} By a `pointed scheme' or a `scheme equipped with a base point' over a base field $K$ we will mean a pair $(X, x)$ where either $X$ is an arbitrary scheme over $K$ and $x\in X(K)$ is a rational point, or $X$ is a smooth curve over $K$ and $x$ is a tangential base point supported at a $K$-point of $\oX\setminus X$ where $\oX$ is the smooth compactification of $X$ (see Section \ref{tangential base points} for a brief review of tangential base points). For both of these settings, $\pi_1^{\et}(X_{\oK},x)$ will denote the \'etale fundamental group of $X_{\oK}$ with respect to the geometric base point supported at $x$. It comes equipped with a continuous action of $G_K$. Likewise, $\pi_1^{\et}(X,x)=G_K\ltimes\pi_1^{\et}(X_{\oK},x)$ will denote the fundamental group of the scheme $X$. 

{\bf Acknowledgements.} I am grateful to Mark Kisin for comments and suggestions on the exposition, to Joseph Ayoub for pointing me to his work \cite{ayoub}, and to the anonymous referee for corrections and suggestions. This research was partially conducted during the period the author served as a Clay Research Fellow, and enjoyed the hospitality of the Max Planck Institute for Mathematics in Bonn.

\section{Pro-algebraic completion}

Let $\Gamma$ be a topological group. For a finite extension $E$ of $\bQ_p$ we denote by $\Gamma^{\pro-\alg}_E$ the pro-algebraic completion of $\Gamma$ over $E$. It is defined as the affine group scheme\footnote{Recall that every affine group scheme over a field is isomorphic to an inverse limit of linear algebraic group schemes} over $E$ equipped with a continuous (with respect to the inverse limit of the $p$-adic topologies on $E$-points of finite type quotients of $\Gamma_E^{\pro-\alg}$) map $\alpha_{\Gamma}:\Gamma\to \Gamma^{\pro-\alg}_E(E)$ satisfying the following universal property. For any continuous homomorphism $\rho:\Gamma\to GL_n(E)$ there exists a unique morphism $\rho^{\alg}:\Gamma^{\pro-\alg}_E\to GL_{n,E}$ of group schemes such that the induced map on $E$-points fits into the commutative diagram

\begin{equation}
\begin{tikzcd}
\Gamma\arrow[rr, "\rho"]\arrow[rd, "\alpha_{\Gamma}"] & & GL_n(E) \\
& \Gamma^{\pro-\alg}_E(E)\arrow[ru, "\rho^{\alg}"] 
\end{tikzcd}
\end{equation}

Similarly, the pro-reductive completion $\Gamma^{\pro-\red}_{E}$ is the pro-reductive group over $E$ satisfying the analogous universal property among representations $\rho:\Gamma\to GL_n(E)$ for which the Zariski closure of the image is a reductive subgroup of $GL_{n,E}$. These notions were first introduced in \cite{hochschild-mostow}, see also \cite{pridham} for a discussion of these objects in a setup very close to ours. This section reviews all the necessary facts about pro-algebraic completions.

Let $\Func^{\cont}(\Gamma, E)$ be the space of all continuous functions $\Gamma\to E$. It is equipped with an action of $\Gamma$ given by $(\gamma\cdot f)(x)=f(\gamma^{-1}x)$ for $\gamma,x\in \Gamma$ and $f\in \Func^{\cont}(\Gamma, E)$.

\begin{lm}\label{pro-alg: functions description}
(i) For a finite extension $E\subset E'$ there is a canonical isomorphism $\Gamma_{E'}^{\pro-\alg}\simeq \Gamma_E^{\pro-\alg}\times_{\Spec E}\Spec E'$.

(ii) The ring of functions $E[\Gamma^{\pro-\alg}_E]$ admits the following description 

\begin{equation}\label{functions on completion equation}
E[\Gamma^{\pro-\alg}_E] = \{f\in\Func^{\cont}(\Gamma, E)|\text{the span of }\Gamma\cdot f\text{ is finite-dimensional over }E\}
\end{equation}
\end{lm}

\begin{proof}
(ii) There is a map $\alpha_{\Gamma}^*:E[\Gamma^{\pro-\alg}_E]\to \Func^{\cont}(\Gamma, E)$ given by precomposing with $\alpha_{\Gamma}$. By definition, a regular function on $\Gamma^{\pro-\alg}_E$ factors through some homomorphism $\Gamma^{\pro-\alg}_E\to GL_{n, E}$. So, to prove that the image of $\alpha_{\Gamma}^*$ is contained in the right-hand side of (\ref{functions on completion equation}) it is enough to observe that for any element $f\in E[GL_{n,E}]$ the orbit $GL_n(E)\cdot f$ spans a finite-dimensional $E$-vector space.

Next, let $f$ be an element of the right-hand side of (\ref{functions on completion equation}). The span of $\Gamma\cdot f$ gives a continuous finite-dimensional representation $V$ of $\Gamma$ and the function $f$ factors through the homomorphism $\Gamma\to GL(V)$, hence it lies in the image of $\alpha_{\Gamma}^*$. Finally, $\alpha_{\Gamma}^*$ is injective  because any regular function on $\Gamma^{\pro-\alg}_E$ factors through an algebraic group and a homomorphism from $\Gamma^{\pro-\alg}_E$ to an algebraic group is completely determined by its restriction to $\Gamma$. 

Part (i) now follows because the right-hand side of (\ref{functions on completion equation}) satisfies base change under finite extensions of $E$.
\end{proof}

\begin{lm}\label{pro-alg: finite index surjection}
If $\Gamma_1\subset \Gamma$ is an open subgroup of finite index then the restriction map $E[\Gamma^{\pro-\alg}_E]\to E[\Gamma_{1,E}^{\pro-\alg}]$ is surjective.
\end{lm}

\begin{proof}
We will use the description of functions on the pro-algebraic completion provided by the right-hand side of (\ref{functions on completion equation}). Let $f_1:\Gamma_1\to E$ be a continuous function whose translates span a finite-dimensional space. Pick representatives for the left cosets of $\Gamma_1\subset \Gamma$ so that $\Gamma=\bigsqcup\limits_{i=1}^dg_i\Gamma_1$ for some $g_2,\dots g_d\in \Gamma$ and $g_1=1$. Then define a function $f:\Gamma\to E$ be declaring $f(g_ih)=f_1(h)$ for every $h\in \Gamma_1$. It evidently extends $f_1$ onto $f$ and its $\Gamma$-translates span a finite-dimensional space.
\end{proof}
\begin{rem}
Lemma \ref{pro-alg: finite index surjection} says that $\Gamma_{1,E}^{\pro-\alg}$ is a closed subgroup scheme of $\Gamma_E^{\pro-\alg}$. The choice of representatives of left cosets $\Gamma/\Gamma_1$ moreover induces a $\Gamma_{1,E}^{\pro-\alg}$-equivariant isomorphism of algebras $E[\Gamma^{\pro-\alg}_E]\simeq E[\Gamma^{\pro-\alg}_{1,E}]\otimes E[\Gamma/\Gamma_1]$. In particular the natural map identifies the quotient $\Gamma_{E}^{\pro-\alg}/\Gamma_{1,E}^{\pro-\alg}$ with the constant finite group scheme $\underline{\Gamma/\Gamma_1}_{E}$.
\end{rem}

\begin{example}\label{pro-alg:cyclic example}Let $\Gamma$ be the infinite cyclic group $\bZ$ endowed with the discrete topology. As implied by the Jordan decomposition, $\Gamma_{\obQ_p}^{\pro-\alg}\simeq \bG_{a,\obQ_p}\times \widehat{\bZ} \times T$ where $T$ is the pro-torus with character group $X^*(T)=\obQ_p^{\times}/\mu_{\infty}$, cf. \cite[Example 1]{bass-lubotzky-magid-mozes}. Here we take $\Gamma_{\obQ_p}^{\pro-\alg}$ to mean the base change $\Gamma_{\bQ_p}^{\pro-\alg}\times_{\bQ_p}\obQ_p$. The proalgebraic completion $\Gamma_{\bQ_p}^{\pro-\alg}$ itself can be described as $\bG_a\times H_0\times T_0$ where $H_0$ is a pro-finite \'etale group scheme corresponding to the $G_{\bQ_p}$-module $\widehat{\bZ}(1)$ and $T_0$ is the (non-split) pro-torus with the $G_{\bQ_p}$-module of characters given by $X^*(T_0)=\obQ_p^{\times}/\mu_{\infty}$.

Similarly, for the pro-finite group $\Gamma=\widehat{\bZ}$ the pro-algebraic completion over $\obQ_p$ can be described as $\bG_{a,\obQ_p}\times\widehat{\bZ}\times T^+$ where $T^+$ is the pro-torus over $\obQ_p$ with the character group $X^*(T^+)=\overline{\bZ}_p^{\times}/\mu_{\infty}$, reflecting the fact that the eigenvalues of a topological generator of $\widehat{\bZ}$ in a continuous representation must belong to $\overline{\bZ}_p$.
\end{example}

We will never work with the pro-algebraic completion in terms of its points but will rather analyze the ring of regular functions on it. Given a continuous representation $\rho:\Gamma\to GL(V)$ on a finite-dimensional $E$-vector space $V$, denote by $\cF(V)$ the $E$-span of the image of the composition $\Gamma\xrightarrow{\rho}GL(V)\subset \End_E V$. The dual space $\cF(V)^{\vee}$ is sometimes referred to as the space of {\it matrix coefficients} of the representation $V$. One might think of the space of functions on $\Gamma^{\pro-\alg}_E$ as of the ring of matrix coefficients of all representations: 

\begin{lm}\label{pro-alg: matrix coefficients}
For every representation $V$, there is a natural embedding $\cF(V)^{\vee}\subset E[\Gamma_{E}^{\pro-\alg}]$ and 

\begin{enumerate}[(i)]
\item $E[\Gamma_{E}^{\pro-\alg}]$ is equal to the union of these subspaces for varying $V$.

\item The space $E[\Gamma_E^{\pro-\red}]$ can be identified with the subspace of $E[\Gamma_E^{\pro-\alg}]$ obtained by taking the union of the subspaces $\cF(V)^{\vee}$ for all semi-simple representations $V$.
\end{enumerate}

\end{lm}

\begin{proof}
The space $(\End_EV)^{\vee}$ of linear functions on the vector space $\End_E V$ maps to $E[\Gamma_E^{\pro-\alg}]$ via restriction to $GL_V$ and pullback along the map $\rho^{\alg}:\Gamma_E^{\pro-\alg}\to GL_V$. Its image in $E[\Gamma_E^{\pro-\alg}]$ is canonically dual to $\cF(V)$. 

Given a function $f\in E[\Gamma_E^{\pro-\alg}]$ denote by $f^{\op}$ the function $f^{\op}(\gamma)=f(\gamma^{-1})$ and let $V$ be the finite-dimensional $E$-span of the orbit of $f^{\op}$ under the action of $\Gamma$. By adjunction, we then obtain a function $\alpha:\Gamma\to V^{\vee}$. Denote by $\ev_1:V\to E$ the functional corresponding to the element $\alpha(1)$.

The function $f$ can be obtained by composing the map $\Gamma\to \End(V)$ with the map $\End(V)\to E$ that sends an endomorphism $A:V\to V$ to $\ev_1(A(f^{\op}))$. Thus $f$ lies in the subspace $\cF(V)^{\vee}\subset E[\Gamma_E^{\pro-\alg}]$. This finishes the proof of part (i).

To show part (ii), note first that the canonical surjection of group schemes $\Gamma_E^{\pro-\alg}\twoheadrightarrow\Gamma_E^{\pro-\red}$ induces an inclusion $E[\Gamma_E^{\pro-\red}]\subset E[\Gamma_E^{\pro-\alg}]$ and for a semi-simple representation $V$ the subspace $\cF(V)^{\vee}$ is contained inside $E[\Gamma_E^{\pro-\red}]$. Conversely, given a function $f\in E[\Gamma_E^{\pro-\red}]$ the above strategy produces a representation $V$ of $\Gamma$ that factors through $\Gamma_E^{\pro-\red}$ because the action of $\Gamma$ via translations on the space $E[\Gamma_E^{\pro-\alg}]$ preserves the subspace $E[\Gamma_E^{\pro-\red}]$ and factors through $\Gamma^{\pro-\red}_E$ on that subspace. Therefore $V$ is semi-simple as a representation of the pro-reductive group $\Gamma^{\pro-\red}_E$ and hence is semi-simple as a representation of $\Gamma$ because the image of $\Gamma$ is Zariski dense in $\Gamma^{\pro-\red}_E(E)$.
\end{proof}

Let $K$ be any base field, with a chosen algebraic closure $\oK\supset K$ and $(X, x)$ be a $K$-scheme equipped with a base point (that is, $x$ is a $K$-point or a tangential base point at infinity). For brevity, we denote the pro-algebraic (resp. pro-reductive) completion of the topological group $\pi_1^{\et}(X,x)$ over $E=\bQ_p$ by $\pi_1^{\pro-\alg}(X,x)$ (resp. $\pi_1^{\pro-\red}(X, x)$). The action of $G_K$ on $\pi_1^{\et}(X_{\oK},x)$ induces an action on the spaces of functions $\bQ_p[\pi_1^{\pro-\alg}(X_{\oK},x)]$ and $\bQ_p[\pi_1^{\pro-\red}(X_{\oK},x)]$.

\begin{example}In general, this action is not locally finite. For instance, consider the case of $X=\bG_{m,K}$ over a field $K$ of characteristic zero containing only finitely many roots of unity. Grothendieck's quasi-unipotent monodromy theorem  comes down to the fact that for a function $f$ on the pro-algebraic completion of $\pi_1^{\et}(X_{\oK},\ox)=\widehat{\bZ} (1)$ the span of its $G_K$-orbit is finite-dimensional if and only if $f$ factors through the canonical map $\widehat{\bZ}_{\obQ_p}^{\pro-\alg}\twoheadrightarrow \bG_{a,\obQ_p}\times\widehat{\bZ}$. 
\end{example}

We will concern ourselves only with the locally finite subspace of $\bQ_p[\pi_1^{\pro-\alg}(X_{\oK},x)]$ which admits an alternative description in terms of the pro-algebraic completion of the arithmetic fundamental group:

\begin{lm}\label{pro-alg: galois finite}
The image of the restriction map $\bQ_p[\pi_1^{\pro-\alg}(X,x)]\to \bQ_p[\pi_1^{\pro-\alg}(X_{\oK},x)]$ coincides with the subspace \begin{multline}\algfun{X_{\oK},x}:=\\ \{f\in \bQ_p[\pi_1^{\pro-\alg}(X_{\oK},x)]\mid \text{the span of }\sigma\cdot f\text{ for }\sigma\in G_K\text{ is finite-dimensional}\}\end{multline}
\end{lm}
\begin{proof}
By Lemma \ref{pro-alg: functions description} (ii) the space of functions $\bQ_p[\pi_1^{\pro-\alg}(X,x)]$ on the pro-algebraic fundamental group can be described as the subspace of the space of all continuous functions $\Func^{\cont}(\pi_1^{\et}(X,x),\bQ_p)$ whose orbit under $\pi_1^{\et}(X,x)=G_K\ltimes \pi_1^{\et}(X_{\oK},x)$ spans a finite-dimensional subspace. Therefore, the image of the restriction map is contained in $\algfun{X_{\oK},x}$.

To prove the converse inclusion, consider the map $\varepsilon:\Func^{\cont}(\pi_1^{\et}(X_{\oK},x))\to \Func^{\cont}(\pi_1^{\et}(X_{\oK},x))$ defined by composing a continuous function with the projection $\pi_1^{\et}(X,x)=G_K\ltimes \pi_1^{\et}(X_{\oK},x)\to \pi_1^{\et}(X_{\oK},x)$ onto the second component of the semi-direct product. For an element $(\sigma, \gamma)\in G_K\ltimes \pi_1^{\et}(X_{\oK},x)$ we have $(\sigma,\gamma)\cdot \varepsilon(f)=\gamma\cdot (\sigma\cdot \varepsilon(f))$. In particular, if $f$ is an element of $\algfun{X_{\oK},x}$ then $\varepsilon(f)$ is a continuous function on the group $\pi_1^{\et}(X,x)$ whose orbit under left translations spans a finite-dimensional space. By construction $\varepsilon(f)|_{\pi_1^{\et}(X_{\oK},x)}=f$, hence for every $f\in \algfun{X_{\oK},x}$ the function $\varepsilon(f)$ gives an element of $\bQ_p[\pi_1^{\et}(X,x)]$ extending $f$, as desired.
\end{proof}

Lemma \ref{pro-alg: matrix coefficients} gives a way to produce elements in $\algfun{X_{\oK},x}$ from local systems on $X$. Viewing a $\bQ_p$-local system $\bL$ on $X$ as a representation of $\pi_1^{\et}(X_{\oK},x)$ on the space $\bL_x$, we get a subspace $\cF(\bL)^{\vee}\subset \bQ_p[\pi_1^{\pro-\alg}(X_{\oK},x)]$. Note that applying the construction $\cF(-)$ to the space $\bL_{x}$ as a representation of the arithmetic fundamental group $\pi_1^{\et}(X,x)$ might potentially yield a larger space but only matrix coefficients of the geometric representation $\pi_1^{\et}(X_{\oK},x)\to GL(\bL_x)$ make a contribution to $\bQ_p[\pi_1^{\et}(X_{\oK},x)]$. Lemma \ref{pro-alg: matrix coefficients} applied to $\Gamma=\pi_1^{\et}(X,x)$ and Lemma \ref{pro-alg: galois finite} imply:

\begin{lm}\label{pro-alg: functions as matrix coeffs}For any local system $\bL$ on $X$ the subspace $\cF(\bL)^{\vee}\subset \bQ_p[\pi_1^{\pro-\alg}(X_{\oK},x)]$ consists of functions locally finite for the $G_K$-action and the space $\bQ_p[\pi_1^{\pro-\alg}(X_{\oK},x)]^{G_K-\fin}$ is the union of such subspaces for varying $\bL$.
\end{lm}

A useful consequence of Lemmas \ref{pro-alg: finite index surjection} and \ref{pro-alg: galois finite} is

\begin{lm}\label{pro-alg: fet surjection}
If $(X,x)\to (Y,y)$ is a finite \'etale cover of $K$-schemes equipped with base points, we get a $G_K$-equivariant surjection $\bQ_p[\pi_1^{\pro-\alg}(Y_{\oK},y)]^{G_K-\fin}\twoheadrightarrow\bQ_p[\pi_1^{\pro-\alg}(X_{\oK},x)]^{G_K-\fin}$. 
\end{lm}

\section{Belyi's theorem and its immediate consequences}

The driving force of all our arguments is the following surprising theorem of Belyi's.

\begin{thm}\label{belyi main}
For a smooth proper geometrically connected curve $C$ over a number field $F$ and a finite set of closed points $S\subset |C|$ there exists a finite morphism $f:C\to\bP^1_F$ such that $f$ is \'etale over $\PF$ and $f(S)\subset \{0,1,\infty\}$. 
\end{thm}

This statement is stronger than \cite[Theorem 4]{belyi} but this is what Belyi's proof actually shows, see also \cite[Theorem 5.4.B]{serremordell}. We will often use the theorem paraphrased in the following way:

\begin{cor}\label{belyi open}
For any smooth, possibly non-proper, curve $U$ over $F$ there exists a dense open subscheme $U'\subset U$ together with a finite \'etale map $U'\to\PF$.
\end{cor}

We will use this result as a black box, except for the proof of Lemma \ref{tensor: telescopic subgroup} which will require us to write down an explicit \'etale cover of $\PF$, using Belyi's proof idea. The following result is an instance of Belyi's theorem implying a universality statement for the Galois action on $\bQ_p[\pi_1^{\pro-\alg}(\PoF,0_v)]$.

\begin{pr}\label{belyi: any pi1}
Suppose that $X$ is a normal quasi-projective scheme of finite type over a number field $F$ and $x\in X(F)$ is a base point lying in the smooth locus of $X$. Any finite-dimensional subquotient of $\bQ_p[\pi_1^{\pro-\alg}(X_{\oF},x)]^{G_F-\mathrm{fin}}$ can be established as a subquotient of  $\bQ_p[\pi_1^{\pro-\alg}(\PoF,0_v)]^{G_F-\fin}$, for any choice of a tangential base point $0_v$ supported at $0$.
\end{pr}

\begin{proof}
We may freely replace the scheme $X$ by another pointed scheme $X', x'\in X'(F)$ admitting a map $f:X'\to X$ that induces a surjection $\pi_1^{\et}(X'_{\oF},\ox')\to \pi_1^{\et}(X_{\oF},\ox)$, as the induced map $\bQ_p[\pi_1^{\pro-\alg}(X_{\oF},\ox)]\to \bQ_p[\pi_1^{\pro-\alg}(X'_{\oF},\ox')]$ is a $G_F$-equivariant embedding. We will use this observation to reduce to the case $\dim X=1$.

The embedding $X^{\sm}\subset X$ of the maximal open smooth subscheme induces a surjection $\pi_1^{\et}(X^{\sm}_{\oF},\ox)\twoheadrightarrow \pi_1^{\et}(X_{\oF},\ox)$ by \cite[Proposition V.8.2]{SGA1}, so we may assume that $X$ is smooth. By the Lefschetz hyperplane theorem for not necessarily proper varieties \cite[Theorem 1.1]{esnault-kindler} we may further assume that $X$ is a (possibly non-proper) curve.

Using Theorem \ref{belyi main} we choose a quasi-finite map $f:X\to \bP^1_F$ that is finite \'etale over $\PF$ and sends $x$ to $0$. Now let $v\in T_0\bP^1_F$ be a non-zero tangent vector for which we want to prove the claim. If $f$ is ramified at $x$, it is not necessarily possible to choose an $F$-rational tangential base point for $X\setminus f^{-1}(\{0,1,\infty\})$ based at $x\in f^{-1}(0)$ that would map to $0_v$ under $f$. If there happens to exist an  $F$-base point $x_w$ such that $f_*(x_w)=0_v$, we can conclude the proof by noticing that Lemma \ref{pro-alg: fet surjection} yields a surjection $\bQ_p[\pi_1^{\pro-\alg}(\PoF,0_v)]^{G_F-\fin}\twoheadrightarrow\bQ_p[\pi_1^{\pro-\alg}(X_{\oF}\setminus f^{-1}\{0,1,\infty\},x_w)]^{G_F-\fin}$, while $\bQ_p[\pi_1^{\pro-\alg}(X_{\oF},x)]^{G_F-\fin}$ embeds into $\bQ_p[\pi_1^{\pro-\alg}(X_{\oF}\setminus f^{-1}\{0,1,\infty\},x_w)]^{G_F-\fin}$.

In general, we can choose such base point $x_{w}$ over $\oF$ and consider the open subgroup $H:=\pi_1^{\et}(X_{\oF}\setminus f^{-1}(\{0,1,\infty\}),x_w)\subset \pi_1^{\et}(\PoF,0_v)$. By Lemma \ref{pro-alg: functions as matrix coeffs}, it is enough to prove that for any local system $\bL$ on $X$ the representation $\cF(\bL)^{\vee}$ (defined with respect to the base point $x$) is a subquotient of $\bQ_p[\pi_1^{\pro-\alg}(\PoF,0_v)]^{G_F-\fin}$. Consider the pushforward $\bL':=f_*(\bL|_{f^{-1}(\PF)})$ which is a local system of rank $\deg f\cdot \rk \bL$ on $\PF$.

The stalk $\bL_x$ embeds canonically into the stalk $\bL'_{0_v}$ and, under the action of $\pi_1^{\et}(\PoF,0_v)$ on $\bL'_{0_v}$, the subgroup $H$ preserves this subspace $\bL_x\subset \bL'_{0_v}$. Moreover, the action of $H$ on $\bL_x$ factors through $H\twoheadrightarrow \pi_1^{\et}(X_{\oF},x)$ with $\pi_1^{\et}(X_{\oF},x)$ acting on $\bL_x$ via the geometric monodromy of the local system $\bL$. 

Let $W\subset\End(\bL'_{0_v})$ be the subspace of operators $A$ that satisfy $A(\bL_x)\subset \bL_x$. The previous paragraph demonstrates that the image of $H$ in $\End(\bL'_{0_v})$ is contained in $W$ and its image under the natural map $W\to \End(\bL_x)$ is equal to $\cF(\bL)$. Therefore $\cF(\bL)$ is a subquotient of $\cF(\bL')$ (the latter defined with respect to the base point $0_v$) and we are done.
\end{proof}

Our proof of the main theorem will require to work simultaneously with all tangential base points supported at $0$ (there is $F^{\times}$ worth of those). Recall the following set of isomorphism classes of finite-dimensional $\bQ_p$-representations of $G_F$ that was mentioned in the introduction

\begin{equation}
\cC_F:=\{V\mid V\text{ appears as a subquotient of }\bQ_p[\pi_1^{\pro-\alg}(\PoF,0_v)]^{G_F-\fin}\text{ for every }v\}
\end{equation}

\begin{cor}\label{belyi: any H1}
For any normal quasi-projective scheme $X$ over $F$ that admits an $F$-rational base point the representation $H^1_{\et}(X_{\oF},\bQ_p)$ belongs to $\cC_F$.
\end{cor}

\begin{proof}
This follows from Proposition \ref{belyi: any pi1} because the canonical map $\pi_1^{\et}(X_{\oF},x)\to H^1_{\et}(X_{\oF},\bQ_p)^{\vee}$ extends to a surjective map from $\pi_1^{\pro-\alg}(X_{\oF},x)$ to the vector group $\underline{H^1_{\et}(X_{\oF},\bQ_p)^{\vee}}$ and the space $H^1_{\et}(X_{\oF},\bQ_p)$ is $G_F$-equivariantly identified with the space of linear functions on that vector group.
\end{proof}

We do not know if the analog of Proposition \ref{belyi: any pi1} is true for an $X$ equipped with a tangential base point $x$, so there potentially might be representations appearing in $\Apio^{G_F-\fin}$ for some, but not all $v$, hence the necessity to work with the class $\cC_F$.

\section{Direct sum and tensor product}

In this section, we show that the class $\cC_F$ is stable under direct sums and tensor products. In particular, any representation from $\cC_F$ appears in $\bQ_p[\pi_1^{\pro-\alg}(\PoF,0_v)]$ with infinite multiplicity. It is important for the argument that we are working with Galois representation simultaneously appearing in $\Apio$ for all choices of the tangential base point at $0$.

For a future application we will prove a slightly stronger statement that allows some freedom in choosing base points:

\begin{pr}\label{tensor: tensor product}
For a tangential base point $0_v$ supported at $0$ there exist two other tangential base points $0_{v_1}$ and $0_{v_2}$ such that, if $V_1$ and $V_2$ are representations of $G_F$ with $V_i$ appearing as subquotients of $\bQ_p[\pi_1^{\pro-\alg}(\PoF,0_{v_i})]^{G_F-\fin}$ for $i=1,2$, then $V_1\otimes V_2$ and $V_1\oplus V_2$ are subquotients of $\Apio^{G_F-\fin}$.

In particular, the class of representations $\cC_F$ is stable under direct sums and tensor products.
\end{pr}

The key to the proof is the following telescopic property of the fundamental group of $\PF$:

\begin{lm}\label{tensor: telescopic subgroup}
There exists an open subgroup $\Gamma\subset\pi_1^{\et}(\PoF,0_v)$ stable under $G_F$ and admitting a $G_F$-equivariant surjection $\Gamma\twoheadrightarrow \pi_1^{\et}(\PoF,0_{v_1})\times\pi_1^{\et}(\PoF,0_{v_2})$ for some tangential base points $0_{v_1},0_{v_2}$ at $0$.
\end{lm}

\begin{proof}
Consider the degree $3$ finite morphism $f:\bP^1_F\to \bP^1_F$ given by $f(z)=\frac{27}{4}z(z-1)^2$. We have $f'(z)=\frac{27}{4}(3z-1)(z-1)$ so the only ramification points of $f$ are $\frac{1}{3}, 1$ and $\infty$. Since $f(1)=0, f(\frac{1}{3})=1,f(\infty)=\infty$ the map $f$ restricts to a finite \'etale cover $\bP^1_{F}\setminus\{0,\frac{1}{3},1,\frac{4}{3},\infty\}\to \PF$. Moreover, since $f$ is unramified at $0$, we may choose a tangential  base point $0_{v_1}$ for $\bP^1_{F}\setminus\{0,\frac{1}{3},1,\frac{4}{3},\infty\}$ such that $f(0_{v_1})=0_v$. 

Define $\Gamma$ as the subgroup $f_*(\pi_1^{\et}(\bP^1_F\setminus\{0,\frac{1}{3},1,\frac{4}{3},\infty\},0_{v_1}))\subset \pi_1^{\et}(\PoF,0_v)$. The inclusion maps $\iota_1:\bP^1_F\setminus\{0,\frac{1}{3},1,\frac{4}{3},\infty\}\to\bP^1_F\setminus\{0,1,\infty\}$, $\iota_2:\bP^1_F\setminus\{0,\frac{1}{3},1,\frac{4}{3},\infty\}\to\bP^1_F\setminus\{0,\frac{1}{3},\frac{4}{3}\}$ induce a $G_F$-equivariant homomorphism $\iota_{1*}\times\iota_{2*}:\Gamma\to \pi_1^{\et}(\PoF,0_{v_1})\times \pi_1^{\et}(\mathbb{P}^1_{\oF}\setminus\{0,\frac{1}{3},\frac{4}{3}\},0_{v_1})$. This homomorphism is surjective: it suffices to check that on corresponding topological fundamental groups as profinite completion preserves surjections. The topological space $\bP^1(\bC)\setminus\{0,\frac{1}{3},1,\frac{4}{3},\infty\}$ is homotopy equivalent to $\bigvee\limits_{i=1}^4 S^1$ in such a way that $\iota_1$ and $\iota_2$ are homotopic to the maps contracting the first two or the last two circles in the bouquet, respectively, which implies the surjectivity.

The map $\iota_{1*}\times\iota_{2*}$ is the desired surjection because $\pi_1^{\et}(\mathbb{P}^1_{\oF}\setminus\{0,\frac{1}{3},\frac{4}{3}\},0_{v_1})$ can be identified via an automorphism of $\bP^1_F$ with $\pi_1^{\et}(\PoF,0_{v_2})$ for some tangential base point $0_{v_2}$. 
\end{proof}

\begin{proof}[Proof of Proposition \ref{tensor: tensor product}] For a base point $0_v$ choose the base points $0_{v_1}, 0_{v_2}$ as directed by Lemma \ref{tensor: telescopic subgroup}. We will start by showing that $V_1\otimes V_2$ is a subquotient of $\Apio^{G_F-\fin}$. The representation $V_1\otimes V_2$ is a subquotient of the following space, where $\Gamma$ is provided by Lemma \ref{tensor: telescopic subgroup}: $$\bQ_p[\pi_1^{\pro-\alg}(\PoF,0_{v_1})]^{G_F-\fin}\otimes\bQ_p[\pi_1^{\pro-\alg}(\PoF,0_{v_2})]^{G_F-\fin}\subset \bQ_p[\Gamma_{\bQ_p}^{\pro-\alg}]^{G_F-\fin}$$ The space $\bQ_p[\Gamma_{\bQ_p}^{\pro-\alg}]^{G_F-\fin}$, in turn, is a quotient of $\Apio^{G_F-\fin}$ by Lemma \ref{pro-alg: fet surjection} so $V_1\otimes V_2$ is a subquotient of $\Apio^{G_F-\fin}$.

Since the representation $V_1\oplus V_2$ is a direct summand of the tensor product $(V_1\oplus \bQ_p)\otimes (V_2\oplus\bQ_p)$, to show that $V_1\oplus V_2$ is a subquotient of $\Apio^{G_F-\fin}$ it suffices to prove that for any tangential point $0_w$ supported at $0$, if a representation $V$ is a subquotient of $\algfun{\PoF,0_w}$ then so is $V\oplus\bQ_p$. This amounts to showing that the trivial representation $\bQ_p^n$ of any dimension $n$ is a subquotient of $\algfun{\PoF,0_w}$.

We can explicitly embed the two-dimensional trivial representation into $\bQ_p[\pi_1^{\et}(\PoF,0_w)]$: the $G_F$-equivariant surjection $\pi_1^{\et}(\PoF,0_w)\to\pi_1^{\et}(\bP^1_{\oF}\setminus\{0,\infty\},0_w)\simeq \widehat{\bZ}(1)\twoheadrightarrow \bZ/2$ induces an embedding $\bQ_p^2\simeq \bQ_p[\bZ/2]\hookrightarrow\bQ_p[\pi_1^{\et}(\PoF,0_w)]$. As we have already proven the first assertion of Proposition \ref{tensor: tensor product} about tensor products, it follows that, as desired, $(\bQ_p^2)^{\otimes n}=\bQ_p^{2^n}$ for every $n$ is a subquotient of $\algfun{\PoF,0_w}$ for every tangential base points $0_w$ supported at $0$. This finishes the proof that if $V_i$ is a subquotient of $\algfun{\PoF,0_{v_i}}$ for $i=1,2$, then $V_1\oplus V_2$ is a subquotient of $\algfun{\PoF,0_v}$. 
\end{proof}

\begin{cor}\label{tensor: cyclotomic character}
$\bQ_p(-1)\in\cC_F$.
\end{cor}

\begin{proof}
Corollary \ref{belyi: any H1} implies that $H^1_{\et}(E_{\oF},\bQ_p)\in \cC_F$ for any elliptic curve $E$. By Poincare duality, $\bQ_p(-1)\simeq H^2_{\et}(E_{\oF},\bQ_p)$, which is a direct summand of $H^1_{\et}(E_{\oF},\bQ_p)^{\otimes 2}$, hence lies in $\cC_F$ as well.
\end{proof}

When running arguments with spectral sequences, we will sometimes implicitly use the following consequence of $\cC_F$ being stable under direct sums.

\begin{cor}
If $V$ is a representation from $\cC_F$ and $\dots \subset F^{i+1}V\subset F^i V\subset \dots $ is a filtration on $V$ then the associated graded representation $\bigoplus\limits_i F^iV/F^{i+1}V$ is also in $\cC_F$.
\end{cor}

\section{Artin motives}

Finding Galois representation attached to $0$-dimensional varieties inside functions on $\pi_1^{\pro-\alg}(\PoF,0_v)$ amounts to unraveling Belyi's argument for the faithfulness of the action of $G_F$ on $\pi_1^{\et}(\PoF,x)$.

\begin{lm}\label{artin motives}
For any finite set $T$ equipped with a continuous action of $G_F$ the representation $\bQ_p[T]$ is a subquotient of $\bQ_p[\pi_1^{\pro-\alg}(\PoF,0_v)]^{G_F-\fin}$ for every tangential base point $0_v$. 
\end{lm}

\begin{proof}
Our plan here is to first prove that for every finite Galois extension $K\supset F$ and for every tangential base point $0_v$ the space $\Apio^{G_F-\fin}$ has some faithful representation of $\Gal(K/F)$ as a subquotient, though it will not yet be guaranteed that there exists a common faithful representation appearing in $\Apio^{G_F-\fin}$ for every base point $0_v$. We will then use Proposition \ref{tensor: tensor product} to deduce that in fact, any finite-dimensional representation of $G_F$ factoring through $\Gal(K/F)$ appears as a subquotient of every $\Apio^{G_F-\fin}$. 

We start by choosing a smooth proper geometrically connected curve $C$ over $K$ that does not descend to any smaller subfield $K'\subset K$. For instance, we can take $C$ to be an elliptic curve over $K$ such that the $j$-invariant $j(C)$ generates the field $K$ over $\bQ$. By Theorem \ref{belyi main} there exists a finite map $f:C\to \bP^1_K$ that is \'etale over $\bP^1_K\setminus\{0,1,\infty\}$. Denote by $U\subset C$ the preimage $f^{-1}(\bP^1_K\setminus\{0,1,\infty\})$. Choosing a tangential $\oF$-base point $x_w$ for $C\setminus U$ that lies above $0_v$, we get an open subgroup $f_*(\pi_1^{\et}(U_{\oK},x_w))\subset\pi_1^{\et}(\bP^1_{\oK}\setminus\{0,1,\infty\}, 0_v)$. If an element $\sigma\in G_F$ stabilizes this subgroup then the scheme $U_{\oK}$ can be descended to the field $(\oF)^{\sigma=1}$. Our choice of $C$ thus forces the stabilizer of this subgroup to be contained inside $G_K\subset G_F$.  In particular, there is a finite $G_F$-equivariant quotient $\pi_1^{\et}(\bP^1_{\oF}\setminus\{0,1,\infty\},0_v)\twoheadrightarrow S$ such that the kernel of the action of $G_F$ on $S$ is contained in $G_K$. All in all, there exists a $G_F$-equivariant finite quotient $\pi_1^{\et}(\PoF,0_v)\to X$ such that the action of $G_F$ on $X$ factors through a faithful action of $\Gal(K/F)$. 

Therefore, for every tangential base point $0_v$ there is a faithful representation $W_v$ of $\Gal(K/F)$ appearing as a subrepresentation of $\Apio$. Since every faithful representation of a finite group $G$ contains a faithful subrepresentation of dimension $\leq|G|$, we may choose the representations $W_v$ in a way that they all belong to finitely many isomorphism classes, as $v$ varies. Let $W_1,\dots, W_N$ be the finite list of these representations.

Fix now a particular tangential base point $0_v$ supported at $0$. Repeatedly applying Proposition \ref{tensor: tensor product}, we can conclude that for any $d\in\bN$ a tensor product of the form $W_1^{\otimes a_1}\otimes\dots\otimes W_N^{\otimes a_N}$, with $a_i\geq d$ for at least one $i$, is a subquotient of $\Apio^{G_F-\fin}$. For a large enough $d$ and any $i$ the tensor power $W_i^{\otimes d}$ contains the regular representation of $\Gal(K/F)$ as a direct summand, since each $W_i$ is a faithful representation. As the tensor product of the regular representation with any representation is a direct sum of copies of the regular representation, $W_1^{\otimes a_1}\otimes\dots\otimes W_N^{\otimes a_N}$ as above contains the regular representation of $\Gal(K/F)$ as a direct summand. Hence the representation $\bQ_p[\Gal(K/F)]$ of $G_F$ belongs to $\cC_F$, and every finite-dimensional representation of $G_F$ factoring through $\Gal(K/F)$ belongs to $\cC_F$.
\end{proof}

\begin{cor}\label{artin: finite extension}
Let $F'\supset F$ be a finite extension. If for a representation $V$ of $G_F$ the restriction $V|_{G_{F'}}$ belongs to $\cC_{F'}$ then $V$ itself is in $\cC_F$.
\end{cor}

\begin{proof}
Choose a tangent vector $v\in T_0\bP^1_F$ and let $0_{v_1},0_{v_2}$ be the corresponding auxiliary tangential base points provided by Proposition \ref{tensor: tensor product}. By assumption, there exists a finite-dimensional subspace $W\subset \bQ_p[\pi_1^{\pro-\alg}(\PoF,0_{v_1})]$ stable under the action of $G_{F'}$ such that $V|_{G_{F'}}$ is a quotient of $W$. Let $W'\supset W$ be the $G_F$-span of $W$ inside $\bQ_p[\pi_1^{\pro-\alg}(\PoF,0_{v_1})]$ which we view as a representation of $G_{F}$. The inclusion $W\subset W'$ gives rise to the inclusion $\Ind_{G_{F'}}^{G_{F}}W\subset \Ind_{G_{F'}}^{G_F}(W'|_{G_{F'}})=W'\otimes\bQ_p[G_F/G_{F'}]$ while $V$ is a quotient of $\Ind_{G_{F'}}^{G_{F}}W$, because the induced representation $\Ind_{G_{F'}}^{G_F}(V|_{G_{F'}})=V\otimes\bQ_p[G_{F}/G_{F'}]$ is. The representation $W'\otimes \bQ_p[G_{F}/G_{F'}]$ is a subquotient of $\Apio^{G_F-\fin}$ by Proposition \ref{tensor: tensor product} and Proposition \ref{artin motives} so $V$ is a subquotient of $\Apio^{G_F-\fin}$, as desired. 
\end{proof}

\section{Dual representations}\label{section duals}

The class $\cC_F$ also turns out to be stable under duality. This is a special feature of tangential base points and the analogs of Proposition \ref{tensor: duality} and Lemma \ref{tensor: any character} for a classical base point in place of $0_v$ are false by Corollary \ref{weil numbers over number field}. These results are not used in the proof of our main theorem but are needed for Lemma \ref{fm restatement}.

\begin{pr}\label{tensor: duality}
If $V\in\cC_F$ then $V^{\vee}\in \cC_F$.
\end{pr}

\begin{proof}
The dual representation $V^{\vee}$ can be written as the tensor product $\Lambda^{\dim V-1}V\otimes (\det V)^{\vee}$ so $V^{\vee}$ is a direct summand of the tensor product $V^{\otimes \dim V-1}\otimes (\det V)^{\vee}$. The character $(\det V)^{\vee}$ belongs to $\cC_F$ by Lemma \ref{tensor: any character} below (the assumption of the lemma is satisfied because $V$ is known to be de Rham at places above $p$ by \cite[Proposition 8.5]{petrov}), so $V^{\vee}$ is also in $\cC_F$ by Proposition \ref{tensor: tensor product}.
\end{proof}

\begin{lm}\label{tensor: any character}
Any continuous character $\chi:G_F\to \obQ_p^{\times}$ that is Hodge-Tate at all places above $p$ is a subquotient of $\obQ_p[\pi_1^{\pro-\alg}(\PoF,0_v)]^{G_F-\fin}$ for every tangential base point $0_v$.
\end{lm}

\begin{proof}
We start by proving that the cyclotomic character $\bQ_p(1)$ embeds into $\Apio$. Let $f:\cE\to \bP^1_{F,\lambda}\setminus\{0,1,\infty\}$ be the Legendre family of elliptic curves over the punctured projective line with coordinate $\lambda$, defined as $\cE=\Proj_{F[\lambda^{\pm 1},(\lambda-1)^{-1}]}F[\lambda^{\pm 1},(\lambda-1)^{-1},x,y,z]/(zy^2-x(x-z)(x-\lambda z))$. Consider the local system $\bL=R^1f_*\bQ_p$ on $\PF$. The geometric local system $\bL|_{X_{\oF}}$ is absolutely irreducible so we may apply the discussion of Example \ref{intro: irreducible monodromy example} to $\bL$.

The restriction of $\bL$ to the punctured formal neighborhood $\Spec F((\lambda))$ of $0\in \bP^1_F$ defines a representation $\rho$ of the Galois group $G_{F((\lambda))}$ on a vector space $W$. The group $G_{F((\lambda))}$ naturally fits into an extension \begin{equation}\label{tensor: fund sequence local}1\to\widehat{\bZ}(1)\simeq G_{\oF((\lambda))}\to G_{F((\lambda))}\to G_F\to 1\end{equation} in which the conjugation action of $G_F$ on $\widehat{\bZ}(1)$ is via the cyclotomic character. Since the geometric monodromy of $\bL$ at the puncture $0$ is unipotent and non-trivial (e.g. by \cite[p. 20]{period-domains}), the invariants of $\widehat{\bZ}(1)$ on $W$ is a $1$-dimensional subspace, which is necessarily stable under the action of all of $G_{F((\lambda))}$. 

Hence $W$ fits into a short exact sequence $0\to \chi_1\to W\to \chi_2\to 0$ where $\chi_1, \chi_2$ are characters of $G_{F((q))}$ factoring through $G_F$. Since $\widehat{\bZ}(1)$ acts non-trivially on $W$, we have a non-zero $G_F$-equivariant map $\widehat{\bZ}(1)\to \Hom_{\bQ_p}(\chi_2,\chi_1)$ sending $g$ to $\rho(g)-\id_W$. This forces $\chi_2$ to be isomorphic to $\chi_1(-1)$.

The choice of a tangential base point $0_v$ defines a splitting $s_v:G_F\to G_{F((\lambda))}$ of the extension (\ref{tensor: fund sequence local}) such that the stalk $\bL_{0_v}$ is a restriction of $W$ along $s_v$. Therefore the representation $\bL_{0_v}$ fits into an extension of the form $0\to \chi_1\to \bL_{0_v}\to \chi_1(-1)\to 0$. By Example \ref{intro: irreducible monodromy example} the Galois representation $\bL_{0_v}\otimes\bL_{0_v}^{\vee}$ can be embedded into $\Apio$. In particular, $\bQ_p(1)=\chi_1\otimes (\chi_1(-1))^{\vee}$ embeds into this space of functions, as desired. This also shows that $\bQ_p(-1)$ is a subquotient of $\Apio^{G_F-\fin}$ (though we already proved this by an easier argument in Corollary \ref{tensor: cyclotomic character}).

By Corollary \ref{belyi: any H1}, for any abelian variety $A$ the representation $H^1_{\et}(A_{\oF},\bQ_p)$ lies in $\cC_F$. Therefore $H^1_{\et}(A_{\oF},\bQ_p)^{\vee}=H^1_{\et}(A^{\vee}_{\oF},\bQ_p)(1)$ is in $\cC_F$ as well. Taking into account that all finite image representations lie in $\cC_F$, we know that $\cC_F$ contains all the objects of the Tannakian subcategory of $\mathrm{Rep}_{\bQ_p}G_F$ generated by \'etale cohomology of CM abelian varieties and finite image representations. By \cite[\S 6]{fontaine-mazur}, this implies that $\cC_F$ contains all abelian representations that are Hodge-Tate at primes above $p$.
\end{proof} 

\section{First cohomology of local systems}

\begin{pr}\label{H1 main}
Let $X$ be any geometrically connected scheme of finite type over a field $K$ equipped with a base point $x$. For a $\bQ_p$-local system $\bL$ on $X$ the Galois representation $H^1_{\et}(X_{\oK},\bL)$ is a subquotient of $\bQ_p[\pi_1^{\pro-\alg}(X_{\oK},x)]^{G_K-\fin}\otimes {\bL_{x}}$.
\end{pr}

\begin{proof}
The first cohomology $H^1_{\et}(X_{\oK},\bL)$ is isomorphic to the first group cohomology $H^1_{\cont}(\pi_1^{\et}(X_{\oK},x), \bL_{x})$ compatibly with the Galois action. The group cohomology is computed by the standard complex $$\bL_{x}\xrightarrow{\partial_0} \Func^{\cont}(\pi_1^{\et}(X_{\oK},x),\bL_{x})\xrightarrow{\partial_1}\dots $$

The subspace $Z^1_{\cont}(\pi_1^{\et}(X_{\oK},x),\bL_{x}):=\ker\partial_1\subset \Func^{\cont}(\pi_1^{\et}(X_{\oK},x),\bL_{x})$ of $1$-cocycles fits into the exact sequence $$0\to H^0(X_{\oK},\bL)\to \bL_{x}\xrightarrow{\partial_0} Z^1_{\cont}(\pi_1^{\et}(X_{\oK},x),\bL_{x})\to H^1_{\et}(X_{\oK},\bL)\to 0$$ Hence $Z^1_{\cont}(\pi_1^{\et}(X_{\oK},x),\bL_{x})$ is a finite-dimensional Galois representation that has $H^1_{\et}(X_{\oK},\bL)$ as a quotient.

On the other hand, as we will now compute, every element $f\in Z^1_{\cont}(\pi_1^{\et}(X_{\oK},x),\bL_{x})$ extends to a function on $\pi_1^{\pro-\alg}(X_{\oK},x)$ with values in the affine scheme corresponding to the vector space $\bL_{x}$. If $f:\pi_1^{\et}(X_{\oK},x)\to \bL_{x}$ is a continuous $1$-cocycle then its translate $f^g$ by an element $g\in \pi_1^{\et}(X_{\oK},x)$ is given by $f^g(h)=f(g^{-1}h)=g^{-1}f(h)+f(g^{-1})$. Therefore, the span of the $\pi_1^{\et}(X_{\oK},x)$-orbit of the function $f$ is contained inside the sum of the finite-dimensional space of constant functions with the space $\langle\rho_{\bL}(\pi_1^{\et}(X_{\oK},x))\rangle\cdot f$  where $\langle\rho_{\bL}(\pi_1^{\et}(X_{\oK},x))\rangle\subset \End(\bL_x)$ is the subalgebra generated by the image of the representation $\rho_{\bL}|_{\pi_1^{\et}(X_{\oK},x)}$. Hence, $Z^1_{\cont}(\pi_1^{\et}(X_{\oK},x),\bL_{x})$ is a subspace of $\bQ_p[\pi_1^{\pro-\alg}(X_{\oK},x)]^{G_K-\fin}\otimes_{\bQ_p}\bL_{x}$ compatibly with the Galois action, so $H^1_{\et}(X_{\oK},\bL)$ is a subquotient of this tensor product.
\end{proof}

\begin{rem}
Another way to see that every $1$-cocycle on $\pi_1^{\et}(X_{\oK},x)$ extends to a function on the pro-algebraic completion is to observe that the canonical map $Z^1_{\alg}(\pi_1^{\pro-\alg}(X_{\oK},x),\bL_x)\to Z^1_{\cont}(\pi_1^{\et}(X_{\oK},x),\bL_x)$ is an isomorphism. This is the case because the source and the target of this map are extensions of $H^1_{\alg}(\pi_1^{\pro-\alg}(X_{\oK},x),\bL_x)$ and $H^1_{\cont}(\pi_1^{\et}(X_{\oK},x),\bL_x)$, respectively, by the space $\bL_{x}/\bL^{\pi_1^{\et}(X_{\oK},x)}_{x}$. The map $H^1_{\alg}(\pi_1^{\pro-\alg}(X_{\oK},x),\bL_x)\to H^1_{\cont}(\pi_1^{\et}(X_{\oK},x),\bL_x)$ is an isomorphism because both groups classify extensions of the trivial representation $\bQ_p$ by $\bL_x$ and the categories of finite-dimensional representations of $\pi_1^{\et}(X_{\oK},x)$ and $\pi_1^{\pro-\alg}(X_{\oK},x)$ are equivalent.
\end{rem}

\section{Proof of Theorem \ref{main theorem intro}}

After the preparatory work of the previous sections, the main result will follow by induction on the dimension, exhibiting the relevant variety as a fibration over a curve and applying a Leray spectral sequence.

\begin{proof}[Proof of Theorem \ref{main theorem intro}]
We will start with some preliminary reductions. The argument can be shortened slightly if we appeal to resolution of singularities but we take care to show that the existence of alterations \cite{dejongalterations} is enough. It is harmless to assume that $X$ is connected and reduced. Next, choose a simplicial $h$-hypercover $Y_{\bullet}\to X$ such that each $Y_i,i\in\bN$ is a smooth $F$-scheme. By cohomological descent \cite[$\mathrm{V}^{\mathrm{bis}}$]{SGA4-2} there is a spectral sequence of Galois representations with $E_1^{ij}=H^j_{
\et}(Y_{i,\oF},\bQ_p)$ converging to $H^{i+j}_{\et}(X_{\oF},\bQ_p)$. Hence any irreducible subquotient of $H^n_{\et}(X_{\oF},\bQ_p)$ appears as an irreducible subquotient of some $H^j_{\et}(Y_{i,\oF},\bQ_p)$, so we may from now on assume that $X$ is smooth.

We will now argue by induction on $\dim X$, the base case $\dim X=0$ being covered by Lemma \ref{artin motives}. If $U\subset X$ is a dense open subscheme then the Gysin sequence and purity imply that any irreducible subquotient of the kernel or the cokernel of the restriction map $H^n_{\et}(X_{\oF},\bQ_p)\to H^n_{\et}(U_{\oF},\bQ_p)$ appears as a subquotient of the representation $H^i_{\et}(Z_{\oF},\bQ_p(-j))$ for some $i,j\geq 0$, and $Z$ a smooth variety with $\dim Z <\dim X$. Therefore establishing the induction step for $X$ is equivalent to doing so for $U$ (recall that by Corollary \ref{tensor: cyclotomic character}, if $H^j_{\et}(Z_{\oF},\bQ_p)\in\cC_F$ then $H^j_{\et}(Z_{\oF},\bQ_p(-j))\in\cC_F$ for $j\geq 0$). Also, we may replace $X$ by a finite \'etale cover $X'\to X$ because, by the Leray spectral sequence, $H^n_{\et}(X_{\oF},\bQ_p)$ is a direct summand of $H^n_{\et}(X'_{\oF},\bQ_p)$.

Next, we will reduce to the case where $X$ admits a smooth proper morphism to a dense open subscheme $\bP^1_F$. We may assume that $X$ is affine and choose a non-constant morphism $f:X\to \bA^1_F$. Choose a possibly singular compactification $\oX\supset X$ and a projective birational morphism $b:\oX'\to \oX$ such that there is a map $\tf:\oX'\to\bP^1_F$ extending $f$ on $b^{-1}(X)\simeq X$. Then choose a smooth alteration $a:\oX''\to \oX'$ as in \cite[Theorem 4.1]{dejongalterations}. There exists an open dense $V\subset \oX''$ that is a finite \'etale cover of an open subscheme of $X$ via the composition $b\circ a$, so it is enough to prove the theorem for $\oX''$. There exists an open dense subscheme $U\subset\bP^1_F$ such that $\tf\circ a$ is smooth over $U$, so we have reduced to proving the theorem for the variety $Y:=f^{-1}(U)$ which admits a smooth proper morphism $\pi:Y\to U$. 

There is a Leray spectral sequence with $E_2^{i,j}=H^i_{\et}(U_{\oF},R^j\pi_*\bQ_p)$ converging to $H^{i+j}_{\et}(Y_{\oF},\bQ_p)$. Therefore, to prove that the semi-simplification of $H^{n}_{\et}(Y_{\oF},\bQ_p)$ is in $\cC_F$, it is enough to prove the same for each of the representations $H^i(U_{\oF},R^j\pi_*\bQ_p)$, because $\cC_F$ is closed under direct sums. By Artin vanishing, the group $H^i_{\et}(U_{\oF},R^j\pi_*\bQ_p)$ can be non-zero only for $i=0$ or $1$. Choose a rational point $x\in U(F)$. By smooth and proper base change theorem each of the sheaves $R^j\pi_*\bQ_p$ is a local system on $U$ and the stalk $(R^j\pi_*\bQ_p)_x$ is isomorphic to the cohomology $H^j_{\et}(f^{-1}(x)_{\oF},\bQ_p)$ of the fiber above $x$. Since $f^{-1}(x)$ is a variety of dimension $<\dim X$, semi-simplifications of the representations $H^j_{\et}(f^{-1}(x),\bQ_p)$ are already known to appear in $\cC_F$, for every $j$. The same immediately follows for the global sections $H^0(U_{\oF},R^j\pi_*\bQ_p)\subset (R^j\pi_*\bQ_p)_x$.

Applying Proposition \ref{H1 main} to the local system $R^j\pi_*\bQ_p$ we see that $H^1_{\et}(U_{\oF},R^j\pi_*\bQ_p)$ is a subquotient of $\bQ_p[\pi_1^{\pro-\alg}(U_{\oF},x)]^{G_F-\fin}\otimes (R^j\pi_*\bQ_p)_x$. By Proposition \ref{belyi: any pi1} the representation $\bQ_p[\pi_1^{\pro-\alg}(U_{\oF},x)]^{G_F-\fin}$ is a union of representations from $\cC_F$ and $(R^j\pi_*\bQ_p)_x$ is in $\cC_F$ by the induction assumption. Since $\cC_F$ is closed under tensor products, the $1$st cohomology group $H^1_{\et}(U_{\oF},R^j\pi_*\bQ_p)$ is in $\cC_F$ as well so the induction step is established.

\end{proof}

\section{Variants and questions}

In this section, we make miscellaneous comments on possible extensions and variations of our main theorem.

\subsection{Frobenius eigenvalues} We start by formulating an analog of Weil's Riemann Hypothesis for fundamental groups that arises from L. Lafforgue's work on the global Langlands correspondence for function fields. These results were proven in \cite[Theorem 1.14, Theorem 1.17]{pridhamweight} in the case of a classical base point. We include the proofs (equivalent to those of Pridham) to highlight the different behaviors that exhibit fundamental groups with respect to classical base points and tangential base points.

\begin{pr}\label{lafforgue finite field}
Let $X$ be a geometrically connected normal variety over a finite field $k$ of characteristic $p$ and $l$ be a prime different from $p$. 

\begin{enumerate}[(i)]
\item If $x$ is any base point of $X$ (that is, a $k$-point or a tangential base point) then the eigenvalues of $\Fr_k$ on both $\bQ_l[\pi_1^{\pro-\red}(X_{\ok},x)]^{G_k-\fin}$ and $\bQ_l[\pi_1^{\pro-\alg}(X_{\ok},x)]^{G_k-\fin}$ are Weil numbers. 
\end{enumerate}

If $x\in X(k)$ is a classical base point then, more specifically, 
\begin{enumerate}[(ii)]
\item The eigenvalues of $\Fr_k$ on $\bQ_l[\pi_1^{\pro-\red}(X_{\ok},\ox)]^{G_k-\fin}$ are Weil numbers of weight $0$.
    \item The eigenvalues of $\Fr_k$ on $\bQ_l[\pi_1^{\pro-\alg}(X_{\ok},\ox)]^{G_k-\fin}$ are Weil numbers of non-negative integral weight. 
\end{enumerate}
\end{pr}

\begin{proof}
We will access the spaces $\bQ_l[\pi_1^{\pro-\red}(X_{\ok},x)]^{G_k-\fin}$ and $\bQ_l[\pi_1^{\pro-\alg}(X_{\ok},x)]^{G_k-\fin}$ through the description of  Lemma \ref{pro-alg: functions as matrix coeffs}. Let $\bL$ be a $\bQ_l$-local system on $X$.

In the situation of (ii), by Lemma \ref{pro-alg: matrix coefficients} (ii), the local system $\bL$ is geometrically semi-simple. It is not necessarily semi-simple on $X$, but replacing $\bL$ by its semi-simplification does not affect Frobenius eigenvalues on $\cF(\bL)$. We can therefore assume that $\bL$ is irreducible and, twisting it by a character of the Galois group $G_k$ we can moreover assume that $\det \bL$ has finite image, by \cite[Theoreme 1.3.1]{deligne-weil2}. By \cite[Proposition VII.7]{lafforgue} the sheaf $\bL$ is then pure of weight $0$ and hence the eigenvalues of $\Fr_k$ on $\cF(\bL)\subset \bL_x\otimes\bL^{\vee}_x$ are Weil numbers of weight zero.  

To deal with (i) and (iii), recall that by \cite[Corollary VII.8]{lafforgue}, the local system $\bL\otimes_{\bQ_l}\obQ_l$ admits a decomposition $\bigoplus\limits_{i=1}^n\chi_i\otimes\bL_i$ where each $\chi_i$ is a $\obQ_l$-character of $G_k$ and $\bL_i$s are mixed $\obQ_l$-local systems on $X$, in the sense of \cite[Definition 1.2.2 (ii)]{deligne-weil2}. Since $\cF(\chi_i\otimes\bL_i)=\cF(\bL_i)$ and $\cF(\bL\otimes_{\bQ_l}\obQ_l)$ embeds into $\bigoplus\limits_{i=1}^n\cF(\chi_i\otimes\bL_i)$, we may assume from the beginning that $\bL$ is a mixed $\obQ_l$-local system.

In other words, there is a filtration $W_{m+1}=0\subset W_m\subset \dots\subset W_n=\bL$ by sub-local systems on $X$ such that each $W_{i}/W_{i+1}$ is pure of weight $(-i)$, cf. \cite[Theoreme 3.4.1 (ii)]{deligne-weil2}. The space of endomorphisms $\End(\bL_{x})$ gets equipped with a $\bZ$-indexed filtration $F_i\End(\bL_{x})=\{A\in \End(\bL_{x})|A(W_j)\subset W_{j+i}\text{ for all }j\}$. The image of the map $\pi_1^{\et}(X_{\ok},x)\to \End(\bL_{x})$ corresponding to $\bL$ lands inside $F_0\End(\bL_{x})$ because the subspaces $W_{j,x}\subset \bL_{x}$ are preserved under the action of $\pi_1^{\et}(X_{\ok},x)$. Each of the quotients $F_i/F_{i+1}$ is identified with $\bigoplus\limits_{j} \Hom(W_{j,x}/W_{j+1,x},W_{i+j,x}/W_{i+j+1,x})$, compatibly with the action of $G_k$. Therefore each $G_k$-representation $F_i/F_{i+1}$ is pure of weight $-i$ and the eigenvalues of $\Fr_k$ on $\cF(\bL)\subset F_0\End(\bL_x)$ are Weil numbers of weights $\leq 0$, as desired.

Finally, to prove (i) it remains to show that for a mixed local system $\bL$ the stalk $\bL_x$ at a tangential base point is a mixed representation of $G_k$. This is a consequence of Deligne's weight monodromy theorem, as stated in \cite[Corollaire 1.8.5]{deligne-weil2}.
\end{proof}

\begin{cor}\label{weil numbers over number field}
Let $X$ be a smooth geometrically connected variety over $F$ equipped with a base point $x$.

\begin{enumerate}[(i)]
\item If $x$ is a tangential base point then for any finite-dimensional $G_F$-representation $V\subset \bQ_p[\pi_1^{\pro-\alg}(X_{\oF},x)]$ there exists a finite set $S$ of places of $F$ such that for every $v\not\in S$ the action of $G_F$ on $V$ is unramified at $v$ and the eigenvalues of the Frobenius element $\Fr_v$ are $\#k(v)$-Weil numbers.

\item If $x$ is a classical base point, we can say more: for any finite-dimensional $G_F$-representation $V\subset \bQ_p[\pi_1^{\pro-\alg}(X_{\oF},x)]$ (resp. $V\subset \bQ_p[\pi_1^{\pro-\red}(X_{\oF},x)]$) there exists a finite set $S$ of places of $F$ such that for every $v\not\in S$ the action of $G_F$ on $V$ is unramified at $v$ and the eigenvalues of the Frobenius element $\Fr_v$ are $\#k(v)$-Weil numbers of non-negative weights (resp. of weight $0$).
\end{enumerate}
\end{cor}

\begin{proof}
The proof is analogous to that of \cite[Corollary 8.6]{petrov}. We will write out the argument for the pro-algebraic completion and the proof for the pro-reductive completion proceeds in the same way. 

Let $f:\pi_1^{\pro-\alg}(X_{\oF},x)\to GL_{n,\bQ_p}$ be a morphism such that $V$ is contained in the image of the induced map $f^*:\bQ_p[GL_{n,\bQ_p}]\to\bQ_p[\pi_1^{\pro-\alg}(X_{\oF},x)]$. The restriction of $f$ to $\pi_1^{\et}(X_{\oF},x)$ necessarily factors through a conjugate of $GL_n(\bZ_p)\subset GL_n(\bQ_p)$ and therefore factors through the pro-$S$ completion $\pi_1^{\et}(X_{\oF},x)\to \pi_1^{\et}(X_{\oF},x)^{(S)}$ for a finite set of primes $S$. Hence $V$ lies in the image of the induced map $\bQ_p[(\pi_1^{\et}(X_{\oF},x)^{(S)})^{\pro-\alg}_{\bQ_p}]^{G_F-\fin}\to\bQ_p[\pi_1^{\pro-\alg}(X_{\oF},x)]^{G_F-\fin}$. 

Enlarging $S$, we may assume that there exists a smooth proper scheme $\ofX$ over $\cO_{F,S}$ equipped with a horizontal normal crossings divisor $\fD\subset\ofX$  such that $X=\fX_F$ for $\fX:=\ofX\setminus\fD$ and $x$ extends to an $\cO_{F,S}$-base point $\tx$ of $\fX$. Choose a place $v$ and an embedding $\oF\subset \oF_v$ yielding a decomposition subgroup $G_{F_v}\subset G_F$. By \cite[Lemma 8.7]{petrov} the space $\bQ_p[(\pi_1^{\et}(X_{\oF},x)^{(S)})^{\pro-\alg}_{\bQ_p}]$ is identified with $\bQ_p[\pi_1^{\et}(\fX_{\overline{k(v)}},\tx_{k(v)})^{(S)})^{\pro-\alg}_{\bQ_p}]$ compatibly with the action of the local Galois group $G_{F_v}$. Therefore the restriction $V|_{G_{F_v}}$ is a subquotient of $\bQ_p[(\pi_1^{\et}(\fX_{\overline{k(v)}},\tx_{k(v)})^{(S)})^{\pro-\alg}_{\bQ_p}]\subset \bQ_p[\pi_1^{\et}(\fX_{\overline{k(v)}},\tx_{k(v)})]$ where the action factors through $G_{F_v}\twoheadrightarrow G_{k(v)}$ and the result follows from Proposition \ref{lafforgue finite field}.
\end{proof}

Thus, a finite-dimensional subrepresentation $V\subset \bQ_p[\pi_1^{\pro-\alg}(X_{\oF},x)]$ not only satisfies the assumptions of the Fontaine-Mazur conjecture but also a potentially (though not actually if the Fontaine-Mazur conjecture is true) stronger condition on the eigenvalues of the Frobenius elements.

Let us explicate how the Fontaine-Mazur conjecture is related to Conjectures \ref{conj geom in pi1} and \ref{conj de rham in pi1}.

\begin{lm}\label{fm restatement}
The Fontaine-Mazur conjecture \cite[Conjecture 1]{fontaine-mazur} is equivalent to the conjunction of Conjecture \ref{conj geom in pi1} and Conjecture \ref{conj de rham in pi1}
\end{lm}

\begin{proof}
Assume that the Fontaine-Mazur conjecture is true.  Conjecture \ref{conj geom in pi1} is implied by the Fontaine-Mazur conjecture because, by \cite[Corollary 8.6]{petrov}, any subquotient of $\obQ_p[\pi_1^{\pro-\alg}(\PoF,0_v)]^{G_F-\fin}$ is geometric in the sense of \cite{fontaine-mazur}. Conjecture \ref{conj de rham in pi1} similarly follows from Theorem \ref{main theorem intro} and Lemma \ref{tensor: any character}, because all the representations in question arise as subquotients of some $H^i_{\et}(X_{\oF},\bQ_p(j))$.

Conversely, an irreducible geometric representation is a subquotient of $\obQ_p[\pi_1^{\pro-\alg}(\PoF,0_v)]^{G_F-\fin}$ for some $0_v$ by Conjecture \ref{conj de rham in pi1}, hence comes from geometry by Conjecture \ref{conj geom in pi1}.
\end{proof}

\subsection{Pro-reductive completion} 
As mentioned in the introduction, our proof of Theorem \ref{main theorem intro} has the disadvantage of appealing to non-semi-simple representations of $\pi_1^{\et}(\PoF,0_v)$. In this section, we discuss partial results on Galois representations appearing inside the space of functions on the pro-reductive completions of fundamental groups. Define the subclass $\cC_F^{\red}\subset\cC_F$ as

\begin{equation}
\cC^{\red}_F:=\{V\mid V\text{ appears as a subquotient of }\bQ_p[\pi_1^{\pro-\red}(\PoF,0_v)]\text{ for every }v\}
\end{equation}

This class shares some of the properties of $\cC_F$:

\begin{pr}\label{cred properties}
\begin{enumerate}[(i)]
\item All representations with finite image belong to $\cC_F^{\red}$

\item If $V_1,V_2\in\cC_F^{\red}$ then $V_1\oplus V_2,V_1\otimes V_2\in\cC_F^{\red}$.

\item If, for a finite extension $F'\supset F$ the restriction $V|_{G_{F'}}$ of a representation $V$ lies in $\cC_{F'}^{\red}$ then $V\in\cC_F$.
\end{enumerate}
\end{pr}

\begin{proof}
The proofs of Lemma \ref{pro-alg: finite index surjection}, Proposition \ref{artin motives}, Proposition \ref{tensor: tensor product} and Corollary \ref{artin: finite extension} go through verbatim with the pro-reductive completion in place of the pro-algebraic completion.
\end{proof}

Notably, the analog of Proposition \ref{H1 main} does not hold for the pro-reductive completion already in the case of the trivial local system $\bL=\underline{\bQ_p}$, as Corollary \ref{weil numbers over number field} shows. We can also describe the class $\cC_{F}^{\red}$ more explicitly using the following

\begin{lm}\label{reductive explicit}
Let $X$ be a geometrically connected scheme over $F$ equipped with a base point $x$. If a finite-dimensional $\obQ_p$-representation $V$ of $G_F$ can be embedded into $\obQ_p[\pi_1^{\pro-\red}(X_{\oF},x)]$ then, for some finite extension $F'\supset F$, the restriction $V|_{G_{F'}}$ is isomorphic to a direct sum of subrepresentations of representations of the form $\bL_x\otimes\bL_x^{\vee}$ where $\bL$ is a geometrically irreducible $\obQ_p$-local system on $X_{F'}$.
\end{lm}

\begin{proof}
We need to prove that if $\bL$ is any geometrically semi-simple local system then the representation $\cF(\bL)$ has the aforementioned form.

Let $\bL|_{X_{\oF}}=\bigoplus\limits_{i\in I}\bM_i$ be the decomposition into irreducible summands. The Galois group $G_F$ then acts continuously on the set of isomorphism classes of $\bM_i$s, so, after replacing $F$ by a finite extension, we may assume that this action is trivial. That is, for each $\sigma\in G_F$ the twist $\bM_i^{\sigma}$ is isomorphic to $\bM_i$.

This implies that each $\bM_i$ extends to a projective representation of $\pi_1^{\et}(X_{F'},x)$ and, by Tate's theorem \cite[Theorem 4]{serre-modforms} (or, alternatively automatically by passing to a finite extension of $F$) each $\bM_i$ in fact extends to a local system $\widetilde{\bM_i}$. We can then consider the caonical map $\Hom_{X_{\oF}}(\widetilde{\bM_i}|_{X_{\oF}},\bL|_{X_{\oF}})\otimes_{}\widetilde{\bM_i}\to\bL$ where $W_i:=\Hom_{X_{\oF}}(\bM_i|_{X_{\oF}},\bL|_{X_{\oF}})=H^0(X_{\oF},(\widetilde{\bM_i}^{\vee}\otimes\bL)|_{X_{\oF}})$ is viewed as a representation of $G_F$. 

Since each $\bM_i$ is irreducible, these maps induce an isomorphism $\bigoplus\limits_{i\in J}W_i\otimes\widetilde{\bM_i}\simeq \bL$ for an appropriate subset $J\subset I$. Since $\cF(\bL_1\oplus\bL_2)$ is a direct summand of $\cF(\bL_1)\oplus\cF(\bL_2)$ for any local systems $\bL_1,\bL_2$ on $X$, we may therefore assume that $\bL=W\otimes\bM$ for some geometrically irreducible $\bM$ on $X$. This finishes the proof  because $\cF(W\otimes\bM)=\cF(\bM)$, and $\cF(\bM)=\bM_x\otimes\bM^{\vee}_x$.
\end{proof}

In the spirit of Theorem \ref{main theorem intro}, geometrically irreducible local systems on any variety give rise to representations in $\cC_F^{\red}$:

\begin{pr}\label{irreducible monodromy}
Let $\bL$ be a geometrically irreducible $\obQ_p$-local system (resp. geometrically absolutely irreducible $\bQ_p$-local system) on a variety $S$ over $F$, equipped with a base point $s\in S(F)$. Then the Galois representation $\bL_s\otimes \bL_s^{\vee}$ is a subquotient of $\obQ_p[\pi_1^{\pro-\red}(\PoF,0_v)]$ (resp. $\bQ_p[\pi_1^{\pro-\red}(\PoF,0_v)]$) for every tangential base point $0_v$. 
\end{pr}

\begin{proof}
Applying the discussion of the Example \ref{intro: irreducible monodromy example} in the introduction, we see that $\bL_s\otimes \bL_s^{\vee}$ is a subrepresentation of $\obQ_p[\pi_1^{\pro-\red}(S_{\oF},s)]$. Proposition \ref{belyi: any pi1}, reproven with pro-reductive completions in place of pro-algebraic completions then implies the claimed result.
\end{proof}

\begin{cor}
If $V=H^1_{\et}(A_{\oF},\obQ_p)$ for an abelian variety $A$ over $F$ or $V=H^2_{\et}(X_{\oF},\obQ_p)$ for a K3 surface $X$ then $V\otimes V^{\vee}\in \cC_F^{\red}$.
\end{cor}

\begin{proof}
Denoting $g=\dim A$ let $S=\cA_{g,\Gamma(3)}$ be the moduli space of principally polarized abelian varieties with full level $3$ structure (the level structure is introduced just to ensure that $\cA_{g,\Gamma(3)}$ is representable by a smooth variety). It is equipped with the universal family $\pi:\cA^{\mathrm{univ}}\to S$. Choosing a basis in $A[3](\oF)$ we get a point $x\in S(F')$  corresponding to $A$ defined over a finite extension $F'\supset F$. The assumption of Proposition \ref{irreducible monodromy} is satisfied for $\bL=R^1\pi_*\bQ_p$ (see e.g. \cite[Lemme 4.4.16]{deligne-hodge2}), so $(V\otimes V^{\vee})|_{G_{F'}}$ is in $\cC_{F'}^{\red}$ and the claim follows by Proposition \ref{cred properties} (iii).

The case of the cohomology of a K3 surface is dealt with in the same way using that the corresponding geometric monodromy representation of the fundamental group of the moduli space is irreducible, cf. \cite[Corollary 6.4.7]{huybrechts}.
\end{proof}

\subsection{Base points}\label{base points subsection} Among the results on the representations appearing in $\bQ_p[\pi_1^{\pro-\alg}(\PoF,0_v)]$ that we have discussed so far, the only one that is genuinely special to tangential base points is Proposition \ref{tensor: duality}, as Corollary \ref{weil numbers over number field} shows. I hope that the proof of Theorem \ref{main theorem intro} can be rectified to show that the semi-simplification of any representation of the form $H^i_{\et}(Y_{\oF},\bQ_p)$, for a variety $Y$ over $F$, is a subquotient of $\bQ_p[\pi_1^{\pro-\alg}(\PoF,x)]^{G_F-\fin}$
for every base point $x$. However, at present, the usage of tangential base points appears to be necessary in the proofs of Proposition \ref{belyi: any pi1} and Proposition \ref{tensor: tensor product}. These difficulties would be remedied if one could answer affirmatively the following general question about Belyi maps.

\begin{question}
Given two points $x, y\in\bP^1(F)\setminus\{0,1,\infty\}$, is it possible to find a finite map $f:\bP_F^1\to\bP_F^1$ that is \'etale above $\PF$ such that $f(0)=0,f(1)=1,f(\infty)=\infty, f(x)=y$?
\end{question}
\section{Tangential base points}\label{tangential base points}

In this section, we recall the notion of a tangential base point at infinity due to \cite[\S 15]{delignetrois} and collect relevant basic facts about it. Let $C$ be a smooth curve over an arbitrary field $F$ of characteristic zero and denote by $\oC$ its smooth proper compactification.

Given a point $x\in (\oC\setminus C)(F)$ and a non-zero tangent vector $v\in T_{\oC,x}$, we may choose a generator $t$ of the maximal ideal $\fm_x\subset\cO_{\oC,x}$ such that the image of $t$ in $\fm_x/\fm_x^2\simeq T_{\oC,v}^{\vee}$ is equal to $1$ when paired with $v$. We will call such $t$ {\it compatible} with the tangent vector $v$. This property defines $t$ uniquely up to multiplication by an element in $1+\fm_x$. The choice of $t$ defines a morphism $\iota:\Spec F((t))\to C$ inducing an isomorphism $\widehat{\cO_{\oC,x}}[1/t]\simeq F((t))$. There is also an embedding $\iota_0:\Spec F((t))\to \Spec F[t,t^{-1}]=\bG_{m,F}$ which is fixed once and for all.  

The {\it tangential base point} $x_v$ associated to $x$ and $v$ is a functor from the category of finite \'etale covers of $C$ to the category of finite \'etale covers of $\Spec F$ defined as the composition

\begin{equation}
\begin{tikzcd}
\FEt(C)\arrow[r, "\iota^*"] & \FEt(\Spec F((t)))\arrow[d, "\sim"] \\
\FEt(\Spec F) & \FEt(\bG_{m,F})\arrow[l, "t=1"]
\end{tikzcd}
\end{equation}

Here the vertical functor is inverse to the restriction along $\iota_0$. The resulting functor does not depend, up to an isomorphism, on the choice of $t$ by \cite[Lemme 15.25]{delignetrois}. If we further choose an algebraic closure $F\subset \oF$ we may define the fundamental groups of $X_{\oF}$ and $X$ with respect to the base point $x_v$, which we denote by $\pi_1^{\et}(X_{\oF},x_v)$ and $\pi_1^{\et}(X,x_v)$, respectively. The latter group can be described as the usual semi-direct product: $\pi_1^{\et}(X,x_v)=G_F\ltimes \pi_1^{\et}(X_{\oF},x_v)$. Fundamental groups defined using tangential base points interact with those defined with respect to classical points as follows:

\begin{lm}\label{tangential: vs classical}
\begin{enumerate}[(i)]
\item Given a point $x\in C(F)$ and a tangent vector $v\in T_xC$, for a finite \'etale cover $U\to C$ the geometric fibers of $U$ at $x$ and of $U|_{C\setminus x}$ at $x_v$ are canonically identified. In particular, there is a natural surjective homomorphism $\pi_1^{\et}(C\setminus x, x_v)\to \pi_1^{\et}(C,x)$.

\item Suppose that $f:D\to C$ is a finite surjective, possibly ramified, morphism between smooth curves. Given a point $x\in D(F)$ and a tangent vector $v\in T_xD$, there exists a tangent vector $w\in T_{f(x)}C$ such that pullback of \'etale covers along $f$ induces a morphism $\pi_1^{\et}(D\setminus f^{-1}(f(x)),x_v)\to \pi_1^{\et}(C\setminus f(x), f(x)_w)$ that is an isomorphism onto an open subgroup.

\item In the situation of (ii), given a tangent vector $w\in T_{f(x)}C$ there exists a tangential base point $x_v$ defined over a finite Kummer extension of $F$ such that $f(x_v)=f(x)_w$.
\end{enumerate} 
\end{lm}

\begin{proof}
(i) This follows directly from the definition because a finite \'etale cover of $\Spec F((t))$ that extends to $\Spec F[[t]]$ is trivial, so the fibers of the corresponding cover of $\bG_{m,K}$ over $0$ and $1$ are canonically identified.

(ii) This is evident if $f$ is unramified at $x$. In general, $f$ induces some morphism $\widehat{\cO}_{C,f(x)}\to\widehat{\cO}_{D,x}$ between completed local rings. Choosing a local coordinate $t$ at $x$ compatible with $v$ and some local coordinate $s$ at $f(x)$ we can write this map as $F((s))\mapsto F((t))$ given by some $s\mapsto a_nt^n+a_{n+1}t^{n+1}+\dots$, with $a_n\neq 0$. The appropriate tangent vector $w$ is then given by $a_n\cdot \frac{\partial}{\partial s}$.

(iii) As in the proof of the previous part, there is an induced morphism $\widehat{\cO}_{C,f(x)}\to\widehat{\cO}_{D,x}$ but this time we choose a local coordinate $s$ for $D$ that is compatible with $w$. If the map between completed local rings is given by $F((s))\to F((t)), s\mapsto a_nt^n+a_{n+1}t^{n+1}+\dots$ with $a_n\neq 0$ then the desired tangent vector $v$ is defined as $a_n^{-1/n}\cdot\frac{\partial}{\partial t}$. 
\end{proof}

\bibliographystyle{alpha}
\bibliography{univ}

\end{document}